\documentclass{amsart}

\makeatletter
 \def\LaTeX{\leavevmode L\raise.42ex
   \hbox{\kern-.3em\size{\sf@size}{0pt}\selectfont A}\kern-.15em\TeX}
\makeatother

\newcommand{\BibTeX}{{\rm B\kern-.05em{\sc
i\kern-.025emb}\kern-.08em\TeX}}

\newtheorem{theorem}{Theorem}[section]

\newtheorem{definition}{Definition}[section]

\newtheorem{remark}{Remark}

\numberwithin{equation}{section}

\begin{document}

\title{Splines and Wavelets on Geophysically Relevant Manifolds}

\maketitle
\begin{center}

\author{Isaac Z. Pesenson }\footnote{ Department of Mathematics, Temple University,
 Philadelphia,
PA 19122; pesenson@temple.edu. The author was supported in
part by the National Geospatial-Intelligence Agency University
Research Initiative (NURI), grant HM1582-08-1-0019. }

\end{center}

{\bf 1. Introduction.}
\bigskip

{\bf 2. Manifolds and operators.}

2.1. Compact Riemannian manifolds without boundary. 

2.2. Compact homogeneous manifolds.

2.3. Bounded domains with smooth boundaries. 

2.4. Radon transform on spheres.

2.5. Hemispherical Radon transform on spheres.

2.6. Radon transform on the group of rotations $SO(3)$.
\bigskip

{\bf 3. Generalized variational splines on compact Riemannian manifolds}

3.1 Generalized interpolating variational splines.

3.2. Approximation by pointwise interpolation and approximation.

3.3. A sampling theorem and a cubature formula. 
\bigskip

{\bf 4. Bandlimited and localized Parseval frames on homogeneous manifolds }

\bigskip
{\bf 5. Applications  to the Radon transform on $S^{d}$ }

5.1. Approximate inversion of the spherical Radon transform using generalized splines.

5.2. A sampling theorem for the spherical Radon transform of bandlimited functions on $S^{d}$.

5.3. Exact formulas for Fourier coefficients  of  a bandlimited function $f$  on $S^{d}$ from  a finite number of samples of  $R f$.

\bigskip

{\bf 6. Applications to the Radon transform on $SO(3)$}

6.1. Approximate inversion of the Radon transform on $SO(3)$ using  generalized splines.

6.2. A sampling theorem for the Radon transform of bandlimited functions on $SO(3)$.

6.3. Exact formulas for Fourier coefficients  of  a bandlimited function $f$  on $SO(3)$ from  a finite number of samples of  $ \mathcal{R} f$.
\bigskip

{\bf 7. Conclusion. }
\bigskip

{\bf References.}

\section{Introduction}
 The interpolation problem on the unit two dimensional sphere $S^{2}$, the
problem of evaluating the Fourier
 coefficients of functions on $S^{2}$  and closely related
problems about approximation and cubature 
 formulas on $S^{2}$ attracted interest of many mathematicians. 
Analysis on $S^{2}$ found
many applications in  seismology, weather prediction, astrophysics, signal analysis,  computer vision, computerized tomography, neuroscience, and statistics. 
	In the last two decades, the importance of  these and other applications triggered the development of various generalized wavelet bases and frames suitable for the unit spheres $S^{2}$ and $S^{3}$ and  the rotation group of $\mathbb{R}^{3}$.  Our list of references is very far from being complete \cite{ANS}-\cite{DH}, \cite{DNW}-\cite{HNSW}, \cite{HMS}-\cite{MNW}, \cite{NW}-\cite{Pes13}. More references can be found in monographs \cite{FGS98}, \cite{FM04}, \cite{LS1}. 
Applications of frames on manifolds to scattering theory,  to statistics and cosmology can be found in  \cite{BKMP},   \cite{GM100},  \cite{M2}-\cite{M-all}, \cite{P}.

Cubature formulas on spheres and interpolation on spheres can be traced back to the classical papers by S. L. Sobolev \cite{Sob} and I. J. Schoenberg \cite{S}.
Substantial and in many ways pioneering work on splines, interpolation and approximation on spheres with many applications to geophysics was done by W. Freeden and his collaborators \cite{F78}-\cite{FS07}.  G. Wahba \cite{W1}-\cite{W2} initiated spherical splines with the aim to advance statistical analysis on spheres. Important construction of the so-called needlets on $S^{2}$ was done by F.J. Narcowich, P. Petrushev and J. Ward \cite{NPW}.

	The goal of the present study is to describe new constructions and applications  of splines and 
	 bandlimited and localized frames in a space $L_{2}(M)$, where $M$ is a compact Riemannian manifold. 
  Our article is a summary of some results for compact manifolds that were obtained in   \cite{BEP}, \cite{BP}, \cite{gpes-1}, \cite{Pes00}-\cite{Pes13}. To the best of our knowledge these are the papers which contain the most general  results about splines, frames and Shannon sampling on compact Riemannian manifolds along with applications to Radon-type transforms on manifolds.

The following classes of manifolds will be considered: compact manifolds without boundary, compact homogeneous manifolds, bounded domains with smooth boundaries in Euclidean spaces. 
	One can think of a manifold as of a surface in a high dimensional Euclidean space. A homogeneous manifold is a surface with "many" symmetries like the sphere $x_{1}^{2}+...+x_{d}^{2}=1$ in Euclidean space $\mathbb{R}^{d}$. An important example of a bounded domain is a ball $x_{1}^{2}+...+x_{d}^{2}\leq 1$  in  $\mathbb{R}^{d}$.

Here is a brief description of the paper. In section \ref{MO} we briefly introduce  situations which are considered in the paper. We introduce some very basic notions which relate to compact Riemannian manifolds and elliptic differential operators on them, compact homogeneous manifolds and Casimir operators, Radon transform (or Funk transform) on the unit sphere $S^{d}$, hemispherical transform on $S^{d}$, Radon transform on the group of rotations $SO(3)$. More details about the Radon transform on general homogeneous manifolds and in particular on $SO(3)$ can be found in \cite{H} and  in \cite{BE}-\cite{BPS},\cite{HPPSS}, \cite{K}, \cite{P}. In section \ref{MO} we formulate two our imprtant results (see \cite{gpes-1} and \cite{pg}): a Theorem \ref{cubature} about positive cubature formulas on general compact Riemnnian manifolds and Theorem \ref{prodthm} about eigenfunctions of a Casimir operator on a compact homogeneous manifold. 

In section \ref{splines} we introduce what we call generalized variational interpolating splines on compact Riemannian manifolds \cite{Pes00}-\cite{Pes07}.
We are motivated by the following problem which is of
interest for integral geometry.
Let $M, \dim M=d,$ be a Riemannian manifold and
 $M_{\nu}, \nu=1,2,...,N,$
 is a family of submanifolds
 $\dim M_{\nu}=d_{\nu}, 0\leq d_{\nu}\leq d.$ Given a
set of numbers $v_{1},v_{2},...,v_{N}$ we would like to find a
function for which
\begin{equation}\label{samples}
\int_{M_{\nu}}fdx=v_{\nu}, \>\>\>\nu=1,2,...,N.
\end{equation}

Moreover, we are interested in a "least curved" function that
satisfies the previous conditions. In other words, we seek a
function that satisfies (\ref{samples}) and minimizes the functional
$$
u\rightarrow \|(1-L)^{t/2}u\|,\>\>\>t\in \mathbb{R},
$$ 
where $L$ is a differential second order elliptic operator which is self-adjoint in the natural space $L_{2}(M)$.
Note that
in the case when the submanifold $M_{\nu}$ is a point the integral
(\ref{samples}) is understood as a value of a function at this point.

Our result is that if $s$ is a solution of such variational
problem then the distribution $(1-L)^{t}s$ should satisfy the
following distributional pseudo-differential equation on $M$ for
any $\psi\in C_{0}^{\infty}(M)$,
\begin{equation}
\int_{M}\psi  (1-L)^{t}s
dx=\sum_{\nu=1}^{N}\alpha_{\nu}\int_{M{\nu}}\psi dx,
\end{equation}
where coefficients $\alpha_{\nu}\in \mathbb{C}$ depend just on
$s$.
This equation allows one to obtain the Fourier coefficients of the
function $s$ with respect to eigen functions of the
operator $L$.

From the very definition our solution $s$ is an "interpolant" in
the sense that it has a prescribed set of integrals. Moreover, we
show that the function $s$ is not just an interpolant but also an
optimal approximation to the set of all functions $f$ in the
Sobolev space $H_{t}(M)$ that satisfy (\ref{samples}) and
\begin{equation}\label{derivative}
\|(1-L)^{t/2}f\|\leq K,
\end{equation}
for appropriate $K >0.$ Namely, we show that $s$ is the center of
the convex and bounded set of all functions that satisfy (\ref{samples}) and
(\ref{derivative}).

In subsection \ref{sampl-cub} we using splines to formulate our generalization of the Shannon's Sampling Theorem for manifolds and to introduce cubature formulas. 

In section \ref{frames} we construct bandlimited and localized  Parseval frames in $L_{2}(M)$ where $M$ is a homogeneous manifold \cite{gpes-1}, \cite{pg}. 

 Let us remind that a  set of vectors $\{\theta_{v}\}$  in a Hilbert space $H$ is called a Hilbert  frame if there exist constants $A, B>0$ such that for all $f\in H$ 
\begin{equation}\label{Frame ineq}
A\|f\|^{2}_{2}\leq \sum_{v}\left|\left<f,\theta_{v}\right>\right|^{2}     \leq B\|f\|_{2}^{2}.
\end{equation}
The largest $A$ and smallest $B$ are called respectively the lower and the upper frame bounds and the ratio $B/A$ is known as the tightness of the frame.  If $A=B$ then $\{\theta_{v}\}$ is a \textit{tight } frame, and if $A=B=1$ it is called a \textit{Parseval } frame. 
Parseval frames are similar in many respects to orthonormal bases.  For example, if all members of a Parseval frame are unit vectors then it  is an  orthonormal basis. 

According to the general theory of Hilbert frames \cite{DS}, \cite{Gr} the frame  inequality (\ref{Frame ineq})  implies that there exists a dual frame $\{\Theta_{v}\}$ (which is not unique in general) 
for which the following reconstruction formula holds 
\begin{equation}
f=\sum_{v}\left<f,\theta_{v}\right>\Theta_{v}.
\end{equation}
The important fact is that in the case of a Parseval frame one can take  $\Theta_{v}=\theta_{v}$.

Using Theorems \ref{cubature} and \ref{prodthm} we construct a Parseval frame  in $L_{2}(M)$ (where $M$ is compact and homogeneous) whose elements are bandlimited and have very strong localization on $M$. Such frame is a substitute for a wavelet basis and it is perfectly suitable to perform multiresolution  analysis on compact homogeneous manifolds.

In section \ref{Ssplines} we apply these results to Radon (Funk) transform $R$  on $S^{d}$ (see \cite{Pes04c}). In this case $M=S^{d}$, every $M_{\nu},\>\>\nu=1,...,N,$ is a great subsphere $w_{\nu}\subset S^{d}$ and $L$ is the Laplace-Beltrami operator on $S^{d}$.  The transform $R$ is invertible on the set of even functions on $S^{d}$ and it transforms even functions into even functions. The objective is to find approximate preimage  of $Rf$ where $f$ is sufficiently smooth and even by using only the set of  values $\{v_{\nu}\}_{1}^{N}$ defined in (\ref{samples}). We achieve the goal (Theorem \ref{Main-Applic}) by constructing an even smooth function $s_{t}(f)$ ($t$ is a smoothness parameter)  which interpolates $f$ in the sense that it has the same set of integrals over subspheres $\{w_{\nu}\}_{1}^{N}$. Although this interpolant is an optimal approximation to $f$ in the sense explained above we are unable to characterize this approximation quantitatively. 

In subsection \ref{pointwise} we explore a different approach. This approach seems to be more complicated but it allows to obtain a natural quantitative estimate on the rate of convergence of our approximations. We consider a set of points $\{x_{\nu}\}_{1}^{N}$ which is dual to the set of subspheres $\{w_{\nu}\}_{1}^{N}$.  Let $\rho>0$ be a separation parameter for the mesh  $\{x_{\nu}\}_{1}^{N}$. The first step now is to use values $\{Rf(x_{\nu})\}_{1}^{N}$ to construct a spline $s_{\rho, \tau}(Rf)$ ($\tau$ is a smoothness parameter) which interpolates and approximates the Radon transform $Rf$. Next, we project $s_{\rho,\tau}(Rf)$ on the space of even smooth functions to obtain even function $\widehat{s}_{\rho, \tau}(Rf)$ which is another (generally a better)  approximation to $Rf$. Now for the even function $S_{\rho, \tau}(f)=R^{-1}\widehat{s}_{\rho, \tau}(Rf)$ we are able to show (see Theorem \ref{last-thm555} ) that a difference  between   $S_{\rho, \tau}(f)$ and  $f$ is of order $\rho$ to a power.  In other words when $\rho$ goes to zero a sequence of corresponding functions $S_{\rho, \tau}(f)$ converges to $f$ with a geometrical rate. Moreover, the Theorem \ref{SSTh} shows that if $f$ is $\omega$-bandlimited and $\rho$ is small enough compare to $\omega$ than $S_{\rho, \tau}(f)$ converges to $f$ when smoothness $\tau$ goes to infinity ( for a sufficiently dense but {\bf fixed} mesh $\{x_{\nu}\}_{1}^{N}$). This result can be considered as a generalization of the Classical Sampling Theorem.

In subsection \ref{exactFCS} we are using Theorems \ref{cubature} and \ref{prodthm} to obtain exact formulas for Fourier coefficients  of  a bandlimited function $f$  on $S^{d}$ from  a finite number of samples of  $R f$.

We note that  analogous results can be obtained for the hemispherical transform on $S^{d}$ and for bounded domains in $\mathbb{R}^{n}$ with smooth boundaries (\cite{Pes04c}, \cite{Pes12}). 

In section \ref{group-applic} results which are similar to the  results of section \ref{Ssplines} obtained for the group Radon transform on $SO(3)$ (see \cite{BEP}, \cite{BP}).

\section{ Manifolds and operators}\label{MO}

We describe all situations which will be discussed in the paper.

\subsection{ Compact Riemannian manifolds without boundary.}\label{no-boundary}

Let $M$ be a compact Riemannian manifold without boundary and   $L$ be a differential second order elliptic operator which is  self-adjoint and negatively semi-definite  in the
space $L_{2}(M)$ constructed using a Riemannian density $dx$. 
The best known example of such operator is the Laplace-Beltrami which is given in a local coordinate system  by the formula
$$
\Delta
f=\sum_{m,k}\frac{1}{\sqrt{det(g_{ij})}}\partial_{m}\left(\sqrt{det(g_{ij})}
g^{mk}\partial_{k}f\right)
$$
where $g_{ij}$ are components of the metric tensor,$\>\>det(g_{ij})$ is the determinant of the matrix $(g_{ij})$, $\>\>g^{mk}$ components of the matrix inverse to $(g_{ij})$.

 In order to have an invertible operator we will work with $I-L$, where $I$ is the identity operator in $L_{2}(M)$.  It is known that for every such operator $L$ the domain of the power $(I-L)^{t/2},\>t\in \mathbb{R}$, is the Sobolev space $H_{t}(M),\>\>t\in \mathbb{R}$. There are different ways to introduce norm in Sobolev spaces. For a fixed operator $L$ we will introduce the graph norm as follows. 
\begin{definition}
\label{Sobolevnorm}
The Sobolev space $H_{t}(M), t\in \mathbb{R}$ can be introduced as the
domain of the operator $(I-L)^{t/2}$ with the graph norm

$$
\|f\|_{t}=\|(I-L)^{t/2}f\|, f\in H_{t}(M).
$$
\end{definition}

Note, that such norm depends on $L$. However, for every two differential of order two elliptic operators  such norms are equivalent for each $t\in \mathbb{R}$. 

Since the operator $(-L)$ is  self-adjoint and positive semi-definite it has a discrete
spectrum $0\leq\lambda_{0}\leq\lambda_{1}\leq \lambda_{2}\leq...,$ and
one can choose corresponding eigenfunctions $u_{0},
u_{1},...$ which form an orthonormal basis of $L_{2}(M).$ A
distribution $f$ belongs to $H_{t}(M), t\in \mathbb{R},$ if and
only if
$$
\|f\|_{t}=
\left(\sum_{j=0}^{\infty}(1+\lambda_{j})^{t}|c_{j}(f)|^{2}\right)^{1/2}<\infty,
$$
where Fourier coefficients $c_{j}(f)$ of $f$ are given by
$$
c_{j}(f)=\left<f,u_{j}\right>=\int_{M}f\overline{u_{j}}.
$$

\begin{definition}\label{bandlimited}
The span of all eigenfunctions of $-L$ whose corresponding eigenvalues are not greater than a positive $\omega$ is denoted by ${\bf E}_{\omega}(L)$ and called the set of $\omega$-bandlimited functions.
\end{definition}

\begin{definition}\label{lattice}
For a sufficiently small $\rho>0$ we will say that a finite set of points $M_{\rho}=\{x_{\nu}\}_{\nu=1}^{N}$ is a
$\rho$-lattice, if
\begin{enumerate}

\item The balls $B(x_{\nu}, \rho/2)$ are disjoint.

\item The balls $B(x_{\nu}, \rho)$ form a cover of $M$.
\end{enumerate}
\end{definition}

The following theorem holds for any compact manifold  \cite{gpes-1}, \cite{pg}.

\begin{theorem}\label{cubature} (Cubature  formula with positive weights)
There exists  a  positive constant $c=c(M)$,    such  that if  $\rho=c\omega^{-1/2}$, then
for any $\rho$-lattice $M_{\rho}=\{x_{\nu}\}_{\nu=1}^{N}$, there exist strictly positive coefficients $\mu_{\nu}>0$ for which the following equality holds for all functions in ${\bf E}_{\omega}(M)$:
\begin{equation}
\label{cubway}
\int_{M}fdx=\sum_{x_{\nu}\in M_{\rho}}\mu_{\nu}f(x_{\nu}).
\end{equation}
Moreover, there exists constants  $\  c_{1}, \  c_{2}, $  such that  the following inequalities hold:
\begin{equation}
c_{1}\omega^{-d/2}\leq \mu_{\nu}\leq c_{2}\omega^{-d/2}, \ d=dim\ M.
\end{equation}
\end{theorem}

\subsection { Compact homogeneous manifolds}\label{homman}

The most complete results will be obtained for compact homogeneous manifolds.

A {\it homogeneous compact manifold} $M$ is a
$C^{\infty}$-compact manifold  on which a compact
Lie group $G$ acts transitively. In this case $M$ is necessary of the form $G/K$,
where $K$ is a closed subgroup of $G$. The notation $L_{2}(M),$ is used for the usual Hilbert spaces,  with  invariant measure  $dx$ on $M$.

The Lie algebra $\textbf{g}$ of a compact Lie group $G$ 
is then a direct sum
$\textbf{g}=\textbf{a}+[\textbf{g},\textbf{g}]$, where
$\textbf{a}$ is the center of $\textbf{g}$, and
$[\textbf{g},\textbf{g}]$ is a semi-simple algebra. Let $Q$ be a
positive-definite quadratic form on $\textbf{g}$ which, on
$[\textbf{g},\textbf{g}]$, is opposite to the Killing form. Let
$X_{1},...,X_{n}$ be a basis of
$\textbf{g}$, which is orthonormal with respect to $Q$.
 Since the form $Q$ is $Ad(G)$-invariant, the operator
\begin{equation}\label{Casimir}
-X_{1}^{2}-X_{2}^{2}-\    ... -X_{n}^{2},    \ n=dim\ G
\end{equation}
is a bi-invariant operator on $G$, which is known as the {\it Casimir operator}.
This implies in particular that the corresponding operator on $L_{2}(M)$,
\begin{equation}\label{Casimir-Image}
\mathcal{L}=-D_{1}^{2}- D_{2}^{2}- ...- D_{n}^{2}, \>\>\>
       D_{j}=D_{X_{j}}, \        n=dim \ G,
\end{equation}
commutes with all operators $D_{j}=D_{X_{j}}$.
The operator $\mathcal{L}$, which is usually called the {\it Laplace operator}, is
the image of the Casimir operator under differential of quazi-regular representation in $L_{2}(M)$. It is important to realize that in general, the operator $\mathcal{L}$ is not necessarily the {\it Laplace-Beltrami operator} of the natural  invariant metric on $M$. But it coincides with such operator at least in the following cases:
1) If $M$ is a $d$-dimensional torus, 2) If the manifold $M$ is itself a compact  semi-simple Lie group group $G$ (\cite{H3}, Ch. II), 3) If $M=G/K$ is a compact symmetric space of
  rank one (\cite{H3},
Ch. II, Theorem 4.11).

The following important result was obtained in \cite{gpes-1}, \cite{pg}.

\begin{theorem}(Product property)
\label{prodthm}
If $M=G/H$ is a compact homogeneous manifold and $L$
 is the same as above, then for any $f$ and $g$ belonging
to ${\bf E}_{\omega}(\mathcal{L})$,  their product $fg$ belongs to
${\bf E}_{4m\omega}(\mathcal{L})$, where $m$ is the dimension of the
group $G$.

\end{theorem}

In the case when $M$ is the rank one compact symmetric space (in particular, the sphere $S^{d}$ or any of the projective spaces $P^{d},\>\mathbb{C}P^{d},\>\mathbb{Q}P^{d}$) there is  a better result.

\begin{theorem}
\label{prodthm-2}
If $M=G/H$ is a compact symmetric space of rank one   then for any $f$ and $g$ belonging
to ${\mathbf E}_{\omega}(\mathcal{L})$,  their product $fg$ belongs to
${\mathbf E}_{2\omega}(\mathcal{L})$.

\end{theorem}

\subsection{Bounded domains with smooth boundaries}\label{Dir}
 
 Our consideration includes also  an open  bounded domain $M\subset \mathbb{R}^{d}$ with a smooth boundary $\Gamma$ which is a smooth $(d-1)$-dimensional oriented manifold.  Let $\overline{M}=M\cup \Gamma$ and $L_{2}(M)$ be  the  space of functions  square-integrable with respect to the Lebesgue  measure $dx=dx_{1}...dx_{d}$ with the norm denoted as $\|\cdot \|$. If $k$ is a natural number the notations $H^{k}(M)$ will be  used for  the Sobolev space of distributions on $M$ with the norm
 $$
 \|f\|_{H^{k
 }(M)}=\left(\|f\|^{2}+\sum _{1\leq |\alpha |\leq k}\|\partial^{|\alpha|} f\|^{2}\right)^{1/2}
 $$
where $\alpha=(\alpha_{1},...,\alpha_{d})$ and $\partial^{|\alpha|}$ is a mixed partial derivative 
$$
\left(\frac{\partial}{\partial x_{1}}\right)^{\alpha_{1}}...\left(\frac{\partial}{\partial x_{d}}\right)^{\alpha_{d}}.
$$
Under our assumptions the space $C^{\infty}_{0}(\overline{M})$ of infinitely smooth functions with support in $\overline{M}$  is dense  
in $H^{k}(M)$. Closure in $H^{k}(M)$ of the space   $C_{0}^{\infty}(M)$ of smooth functions with support in $M$ will be denoted as $H_{0}^{k}(M)$.

 Since $\Gamma$ can be treated as a smooth Riemannian manifold one can introduce Sobolev scale of spaces $H^{s}(\Gamma),\>\> s\in \mathbb{R},$ as, for example, the domains of the Laplace-Beltrami operator $L$ of a Riemannian metric on $\Gamma$.
 
 According to the trace theorem there exists a well defined continuous surjective trace operator 
 $$
 \gamma: H^{s}(M)\rightarrow H^{s-1/2}(\Gamma),\>\>s>1/2,
 $$
 such that for all functions $f$  in $H^{s}(M)$ which are smooth up to the boundary the value $\gamma f$ is simply a restriction of $f$ to $\Gamma$.
 One considers the following  operator 
 \begin{equation}\label{Op}
 Pf=\sum_{j,k}\partial_{j}\left(a_{j,k}(x)\partial_{k}f\right),
\end{equation}
with coefficients in $C^{\infty}(M)$ where the matrix $(a_{j,k}(x))$ is real, symmetric  and negatively  definite on $\overline{M}$.
 The operator $L$ is defined as the Friedrichs extension of $P$, initially defined on $C_{0}^{\infty}(M)$, to the set of all functions $f$ in $H^{2}(M)$ with constraint $\gamma f=0$. The Green formula implies that this operator is self-adjoint. It is also a positive operator and the domain of its positive square root $(-L)^{1/2}$ is the set of all functions $f$ in $H^{1}(M)$ for which $\gamma f=0$. 
   Thus, one has a self-adjoint positive definite operator $(-L)$  in the Hilbert space $L_{2}(M)$ with a discrete spectrum $0<\lambda_{1}\leq \lambda_{2},...$ which goes to infinity.  

One of the most important examples of such domain and operator is the unit ball in $\mathbb{R}^{d}$ with the regular Laplace operator on it. In this case the eigenvalues are 
 $
j^{2}_{d+\frac{k-2}{2},\>l},
 $
 where $j_{\nu,\>l}$ is $l$-th positive root of  the Bessel function $J_{\nu}$ of first kind and of order $\nu$.
The set of eigenfunctions given in spherical coordinates by the formulas
 \begin{equation}\label{EF}
 \varphi_{d,\>i,\>l}=c_{d,i,l}\rho^{-\frac{k-2}{2}}J_{d+\frac{k-2}{2}}\left(j_{d+\frac{k-2}{2}, \>l}\>\>\rho\right)Y^{i}_{d,k}(\vartheta),
\end{equation}
 where $k=0,1,...,\>\>1\leq i\leq n_{d}(k),\>\> l=1,2,....$. Constants $c_{d,\>i,\>l}$ can be chosen in a way that makes functions $\varphi_{d,\>i,\>l}$ normal.

\subsection{Radon transform on spheres}\label{Radon-spheres}

 We consider the unit sphere
$S^{d}\subset \mathbb{R}^{d+1}$ and the corresponding space
$L_{2}(S^{d})$ constructed with respect to normalized and
rotation-invariant measure. 

Let  $Y^{i}_{k}$ be an orthonormal basis of spherical
harmonics in the space $L_{2}(S^{d})$, where $k=0,1, ... ; i=
1,2,..., n_{d}(k)$ and
$$
n_{d}(k)=(d+2k-1)\frac{(d+k-2)!}{k!(d-1)!}
$$
is the dimension of the subspace of spherical harmonics of degree
$k$. 
Note, that 
\begin{equation}\label{symm}
Y_{k}^{i}(-x)=(-1)^{k}Y_{k}^{i}(x).
\end{equation}
The Fourier decomposition of $f\in L_{2}(S^{d})$ is
\begin{equation}\label{Fourier-Laplace}
f=\sum_{i,k}c_{i,k}(f)Y^{i}_{k},
\end{equation}
where
$$
c_{i,k}(f)=\int_{S^{d}}f\overline{Y^{i}_{k}}dx=\left<f,Y_{k}^{i}\right>_{L_{2}(S^{d})}.
$$
 To every function  $f\in L_{2}(S^{d})$ the spherical Radon transform associates  its integrals over great subspheres:
$$
Rf(\theta^{\perp}\cap S^{d})=\int_{\theta^{\perp}\cap S^{d}}fdx,
$$
where $ \theta^{\perp}\cap S^{d}$ is the great subsphere of $S^{d}$
whose plane has normal $\theta$.

If a function $f\in L_{2}(S^{d})$ has Fourier coefficients
$c_{i,k}(f)$ then its Radon Transform is given by the formula
$$
R(f)=\pi^{-1/2}\Gamma((d+1)/2)\sum_{i,k}r_{k}c_{i,k}(f)Y_{k}^{i}.
$$
where $Y_{k}^{i}$ are the spherical harmonic polynomials and
\begin{equation}\label{r}
r_{k}=(-1)^{k/2}\Gamma((k+1)/2))/\Gamma((k+d)/2)
\end{equation}
if $k$ is even and $r_{k}=0$ if $k$ is odd. It implies in particular that operators $\Delta$ and $R$ commute on a set of smooth functions.  A function $f \in L_{2}(S^{d})$ is said to be even if its Fourier series (\ref{Fourier-Laplace}) contains only harmonics of even degrees $k=2m$.
Because the coefficients $r_{k}$ have asymptotics
$(-1)^{k/2}(k/2)^{(1-d)/2}$ as $k$ goes to infinity we have the following result.

\begin{theorem}\label{SRT}
The spherical Radon transform $R$ is a continuous operator from the Sobolev space of even functions
$H^{even}_{t}(S^{d})$ onto the space $H^{even}_{t+(d-1)/2}(S^{d})$.
Its inverse $R^{-1}$ is a continuous operator from the space
$H^{even}_{t+(d-1)/2}(S^{d})$ onto the space $H^{even}_{t}(S^{d})$.  If $f\in H^{even}_{t+(d-1)/2}(S^{d})$ and it has Fourier series $f=\sum_{i, m}c_{i, 2m}(f)Y_{2m}^{i}$ then 
\begin{equation}\label{inverse}
R^{-1}f=\frac{\sqrt{\pi}}{\Gamma((d+1)/2)}\sum_{i,m}\frac{c_{i,2m}(f)}{r_{2m}}Y_{2m}^{i}.
\end{equation}

\end{theorem}

\subsection{Hemispherical Radon transform on $S^{d}$}\label{hem}

 To every function  $f\in L_{2}(S^{d})$ the hemispherical transform $T$
 assigns a function $Tf\in L_{2}(S^{d})$ on the dual sphere $S^{d}$
 which is given
by the formula
$$
(Tf)(\xi)=\int_{\xi\cdot x >0} f(x)dx.
$$

For every function $f\in L_{2}(S^{d})$ that has Fourier
coefficients $c_{i,j}(f)$ the hemispherical transform can be given
explicitly by the formula
$$
Tf=\pi^{(d-1)/2}\sum_{i,k}m_{k}c^{i}_{k}(f)Y^{i}_{k},
$$
where $m_{k}=0, $ if $k$ is even and
$$
m_{k}=(-1)^{(k-1)/2}\frac {\Gamma(k/2)}{\Gamma((k+d+1)/2))},
$$
 if $k$ is odd.

The transformation $T$ is one to one on the subspace of odd
functions (i.e. $f(x)=-f(-x))$ of a Sobolev space
$H_{t}^{odd}(S^{d})$ and maps it continuously onto
$H_{t+(d+1)/2}^{odd}(S^{d})$,
$$
T(H_{t}^{odd})(S^{d})=H_{t+(d+1)/2}^{odd}(S^{d}).
$$

\subsection{Radon transform on the group of rotations  $SO(3)$}

The group  of rotations $SO(3)$ of $\mathbb{R}^{3}$ consists of $3\times 3$ real matrices $U$ such that  $U^TU = I,\>\>\ {\rm det\,}U = 1$.
It is known that any $g\in SO(3)$ has a unique representation of the form
$$ 
g = Z(\gamma)X(\beta)Z(\alpha),\ 0\leq \beta \leq \pi,\ 0\leq \alpha,\,\gamma < 2\pi,
$$
 where 
$$ Z(\theta) = \left(\begin{array}{ccc} \cos \theta & -\sin \theta & 0 \\ \sin \theta & \cos \theta & 0 \\ 0 & 0 & 1 \end{array}\right), \quad \mbox X(\theta) = \left(\begin{array}{ccc} 1 & 0 & 0 \\ 0 & \cos \theta & -\sin \theta  \\ 0 & \sin \theta & \cos \theta \end{array}\right) 
$$
are  rotations about the $Z$- and $X$-axes, respectively. In the coordinates $\alpha, \beta, \gamma$, which are known as Euler angles, the Haar measure  of the group $SO(3)$ is given as 
$$
dg= \frac{1}{8\pi^2} \sin\beta d\alpha\,d\beta\,d\gamma. 
$$ 
In other words the following formula holds:
$$ \int_{SO(3)} f(g)\,dg = \int_0^{2\pi}\int_0^{\pi}\int_0^{2\pi} f(g(\alpha,\,\beta,\,\gamma))\frac{1}{8\pi^2} \sin\beta d\alpha\,d\beta\,d\gamma.
$$ 
Note that if $SO(2)$ is the group of rotations of $\mathbb{R}^{2}$ then the two-dimensional sphere $S^{2}$ can be identified with the factor $SO(3)/SO(2)$.

We introduce Radon transform $\mathcal{R} f$ of a smooth function $f$ defined on $SO(3)$.\begin{definition}\label{groupRT}
If $S^{2}$ is the standard unit sphere in $\mathbb{R}^{3}$ , then for a pair $(x,y)\in S^{2}\times S^{2}$ the value of the  Radon transform $\mathcal{R} f$ at $(x,y)$ is defined by the formula
\begin{equation}\label{RRR}
 (\mathcal{R} f)(x,y)   = \frac{1}{2\pi} \int_{\{g\in SO(3): x=gy\}} f(g) d\nu_g = 
 $$
 $$
 4\pi \int_{SO(3)} f(g)\delta_y(g^{-1}x) dg = (f*\delta_y)(x),\>\>\> (x,y)\in S^{2}\times S^{2},
\end{equation}
where $d\nu_g = 8\pi ^2 dg, $  and  $\delta_{y}$ is the measure concentrated on the set of all $g\in SO(3)$ such that $x=gy$. 
\end{definition}

An orthonormal system in $L_2(S^2)$ is provided by the spherical harmonics $\{Y_k^i,\,k\in \mathbb{N}_0,\ i=1,\ldots , 2k+1\}.$ The subspaces  $\mathcal{H}_k:={\rm span}\,\{Y_k^i, i=1,\,\ldots ,\,2k+1\}$ spanned by the spherical harmonics of degree $k$ are the invariant subspaces of the quasi-regular representation 
$ T(g):\,f(x)\mapsto f(g^{-1}\cdot x), $
(where $\cdot $ denotes the canonical action of $SO(3)$ on $S^2$). Representation $T$ decomposes into $(2k+1)$-dimensional irreducible representation $\mathcal{T}_k$ in $\mathcal{H}_k.$
The corresponding matrix coefficients are the Wigner-polynomials
$$ \mathcal{T}_k^{ij}(g) = \langle \mathcal{T}_k(g)Y_k^i, Y_k^j \rangle. $$
If $\Delta_{SO(3)}$ and $\Delta_{S^{2}}$ are Laplace-Beltrami operators of invariant metrics on $SO(3)$ and $S^{2}$ respectively, then 
\begin{equation}\label{eigenvalues}
 \Delta_{SO(3)}\mathcal{T}_k^{ij} = -k(k+1)\mathcal{T}_k^{ij}\quad \mbox{and}\quad \Delta_{S^2}Y_k^i = -k(k+1)Y_k^i. 
\end{equation}
 Using  the fact that $\Delta_{SO(3)}$ is equal to $-k(k+1)$ on the eigenspace $\mathcal{H}_k$ we obtain 
$$ ||f||^2_{L_2(SO(3))} =  \sum_{k=1}^{\infty} (2k+1) ||(4\pi)^{-1}\hat{f}(k)||^2_{L_2(S^2\times S^2)} = 
$$
$$
||(4\pi)^{-1}(I-2\Delta_{S^{2}\times S^{2}})^{1/4}\mathcal{R} f||^2_{L_2(S^2\times S^2)} , $$
where $\Delta_{S^{2}\times S^{2}}=\Delta_1 + \Delta_2 $ is the Laplace-Beltrami  operator of the natural metric on $S^2\times S^2.$ 
We define the following norm on the space  $ C^{\infty}(S^2\times S^2)$
$$
 ||| u |||^2 = ((I-2\Delta_{S^2\times S^2} )^{1/2}u,\,u)_{L_2(S^2\times S^2)}.
 $$
Because $\mathcal{R}$ is essentially an isometry between $L_{2}(SO(3))$ with the natural norm and $L_{2}(S^{2}\times S^{2})$ with the norm $|||\cdot|||$ the inverse of $\mathcal{R}$  is given by  its  adjoint operator. 
To calculate the adjoint operator we express the Radon transform $\mathcal{R}$ in another way. 
Going back to our problem in crystallography we first state that the great circle $C_{x,y}=\{g\in SO(3): g\cdot x = y\}$ in $SO(3)$ can also be described by the following formula
$$ C_{x,y}= x^{\prime}SO(2)(y^{\prime})^{-1}  := \{x^{\prime}h(y^{\prime})^{-1},\ h\in SO(2)\}, \quad x^{\prime},\,y^{\prime}\in SO(3), $$
where $x^{\prime}\cdot x_0 = x,\ y^{\prime}\cdot x_0 = y$ and $SO(2)$ is  the stabilizer of $x_0\in S^2.$ Hence,
\begin{eqnarray*}
\mathcal{R} f(x,y) = \int_{SO(2)} f(x^{\prime}h(y^{\prime})^{-1})\, dh = 4\pi \int_{C_{x,y}} f(g)\, dg  \\ =
 4\pi \int_{SO(3)} f(g)\delta_y(g^{-1}\cdot x)\, dg, \quad f\in L_2(SO(3)). 
 \end{eqnarray*}
 By using this representation one can find that the $L^2$-adjoint operator of $\mathcal{R}$ is given by
\begin{eqnarray}\label{adjoint_op}
  \mathcal{R}^*u = (4\pi) \int_{S^2} (I-2\Delta_{S^2\times S^2})^{1/2}u(g\cdot
y,\,y)\,dy.
\end{eqnarray}

\begin{definition}[Sobolev spaces on $S^2\times S^2$] The Sobolev space 
  $H_{t}(S^2\times S^2),\,t\in \mathbb{R},$ is defined as the domain of the operator 
  $(I-2\Delta_{S^2\times S^2})^{\tfrac{t}{2}}$ with graph norm
		$$ ||f||_t = ||(I-2\Delta_{S^2\times S^2})^{\tfrac{t}{2}}f||_{L^2(S^2\times S^2)} ,
		$$
and 
the Sobolev space $H_t^{\Delta}(S^2 \times S^2),\,t\in \mathbb{R},$ is defined as the
 subspace of all functions $f\in H_t(S^2\times S^2)$ such $\Delta_1 f = \Delta_2 f.$ 
\end{definition}

\begin{definition}[Sobolev spaces on $SO(3)$] The Sobolev space $H_t(SO(3)),\,t \in \mathbb{R},$ is defined as the domain of the operator $(I-4\Delta_{SO(3)})^{\tfrac{t}{2}}$ with graph norm
			$$ |||f|||_t = ||(I-4\Delta_{SO(3)})^{\tfrac{t}{2}}f||_{L^2(SO(3))},\>\>f\in L_2(SO(3)). $$
\end{definition}

It is not difficult to prove the following theorems.
\begin{theorem}\label{mapRT}
    For any $t\geq 0$ the Radon transform on $SO(3)$ is an invertible  mapping
        \begin{align}
           \mathcal{R} : H_{t}(SO(3))\to H_{t+\frac{1}{2}}^{\Delta}(S^2\times S^2).
        \end{align}
     and
$$
  f(g) = \int_{S^2} (I-2\Delta_{S^2\times S^2})^{\tfrac{1}{2}}(\mathcal{R} f)(gy, y) dy = \frac{1}{4\pi} (\mathcal{R} ^* \mathcal{R} f )(g),\>\>\>g\in SO(3).
$$
 Thus, $\mathcal{R}^{-1}=\frac{1}{4\pi}\mathcal{R}^{*}. $
\end{theorem}
One can verify that the following relations hold
\begin{equation}\label{basis-action1}
   \mathcal{R}\mathcal{ T}^k_{ij}(x,y) = \mathcal{T}_{i1}^k(x) \overline{\mathcal{T}^k_{j1}(y)} =\frac{4\pi}{2k+1} Y_k^i(x)\overline{Y_k^j(y)},
\end{equation}
\begin{equation}\label{I1}
\Delta_{S^{2}\times S^{2}}\mathcal{R}f=2\mathcal{R}\Delta_{SO(3)}f,\>\>\>f\in H_{2}(SO(3)), 
\end{equation}
\begin{equation}\label{I2}
\left(1-2\Delta_{S^{2}\times S^{2}}\right)^{t/2}\mathcal{R}f=\mathcal{R}\left(1-4\Delta_{SO(3)}\right)^{t/2}f,\>\>\>f\in H_{t}(SO(3)), \>\>\>t\geq 0,
\end{equation}
\begin{equation}\label{I3}
\left(1-2\Delta_{S^{2}\times S^{2}}\right)^{t/2}g=\mathcal{R}\left(1-4\Delta_{SO(3)}\right)^{t/2}\mathcal{R}^{-1}g,
\end{equation}
where $g\in H_{t+1/2}^{\Delta}(S^{2}\times S^{2}), \>\>\>t\geq 0.$
\begin{theorem}[Reconstruction formula]\label{Recon}
Let
\begin{align*}
    G(x,y) = \mathcal{R} f(x,y) &= \sum_{k=0}^\infty \sum_{i,j=1}^{2k+1} \widehat G(k)_{ij} Y_k^i(x) 
    \overline{Y_k^j(y)} \in H_{\frac{1}{2}+t}^{\Delta}(S^2\times S^2),\
    t\geq 0,
\end{align*}
be a result of the Radon transform. Then the pre-image $f\in H_t(SO(3)),\ t\geq 0,$ is given by
\begin{align*}
    f &= \sum_{k=0}^\infty \sum_{i,j=1}^{2k+1} \frac{(2k+1)}{4\pi} \widehat G(k)_{ij} 
    \mathcal{T}_{ij}^k = \sum_{k=0}^\infty \sum_{i,j=1}^{2k+1} (2k+1)\widehat f(k)_{ij} 
    \mathcal{T}_{ij}^k  \\
    & = \sum_{k=0}^\infty (2k+1){\rm trace\,}(\widehat f(k) \mathcal{T}^k).
\end{align*}
\end{theorem}

\section{Generalized variational splines  on compact Riemannian manifolds}\label{splines}

\subsection{Generalized interpolating variational splines}For a given finite family of pairwise different submanifolds $\{M_{\nu}\}_{1}^{N}$ consider the following family of  distributions
 \begin{equation}
 \label{functionals}
 F_{\nu}(f)=\int_{M_{\nu}}f
 \end{equation}
 which are well defined at least for functions in $H_{\varepsilon +d/2}(M),\>\>\varepsilon>0$.
 In particular, if $M_{\nu}=x_{\nu}\in M$, then every $F_{\nu}$ is a Dirac measure $\delta_{x_{\nu}}\>\>\nu=1,...,N,\>\>x_{\nu}\in M.$
 
 Note that distributions $F_{\nu}$ belong to $H_{-\varepsilon-d/2}(M)$ for any $\varepsilon>0$.
  Given a sequence of complex numbers $v=\{v_{\nu}\},$ $ \nu=1,2,...,N,$ and a $t>d/2$ we consider the
following variational problem:

\textsl{Find a function  $u$ from the space $H_{t}(M)$ which has
the following properties:} \label{var_prob}

\begin{enumerate}

\item $ F_{\nu}(u)=v_{\nu}, \>\>\>\nu=1,2,...,N,\>\>\> v=\{v_{\nu}\},$

\item  $u$ \textsl{minimizes functional $u\rightarrow \|(1-L)
^{t/2}u\|$.}

\end{enumerate}

We show that the solution to Variational problem exist and is
unique for any $t>t_{0}$.  

We need
the following  Independence  Assumption in order to determine the Fourier
coefficients of the solution.

 \textbf{Independence Assumption.} \textsl{There are functions
  $\vartheta_{\nu}\in C^{\infty}(M)$ such
that}
\begin{equation}
\label{independence}
F_{\nu}(\vartheta_{\mu})=\delta_{\nu\mu},
\end{equation}
\textsl{where $\delta_{\nu\mu}$ is the Kronecker delta.}

Note, that this assumption implies in particular that the
functionals $F_{\nu}$ are linearly independent. Indeed, if for certain coefficients $\gamma_{1},\gamma_{2}, ...,
\gamma_{N}$ we have a relation 
$
\sum _{\nu=1}^{N}\gamma_{\nu}F_{\nu}=0,
$
then for any $1\leq\mu\leq N$ we obtain that 
$
0=\sum_{\nu=1}^{N}\gamma_{\nu}F_{\nu}(\vartheta_{\mu})=\gamma_{\mu}.
$

The families of distributions that satisfy our condition
  include in particular finite families of $\delta$ functionals and their
  derivatives. Another example is a set  of integrals over submanifolds from a   finite family of
submanifolds of any codimension.

 The solution to the Variational
Problem will be called a spline and will be denoted as $s_{t}(v).$
 The set of all solutions for a fixed set of distributions
$F=\{F_{\nu}\}$ and a fixed $t$ will be denoted as $S(F,t).$
\begin{definition}
\label{interpolationdef}
Given a function $f\in H_{t}(M)$ we will say that 
spline $s\in S(F,t)$ interpolates $f$  if
$$
F_{\nu}(f)=F_{\nu}(s).
$$
\end{definition}
Interpolating  spline exists and unique (see below) and  will be denoted as $s_{t}(f).$
Note, that from the point of view of the classical theory of variational
 splines it would be more natural to consider minimization of the
functional
$
u\rightarrow \|L^{t/2}u\|.
$
However, in the case of a general compact manifolds it is easer to
work with the operator $I-L$ since this operator is
invertible. 

Our main result concerning  variational splines is the following (see \cite{Pes04c}).
\begin{theorem}
\label{MainTheorem}
If $t>d$, then for any given sequence of scalars 
$v=\{v_{\nu} \}, \nu=1,2,...N,$ the following statements are
equivalent:

\begin{enumerate}

\item $s_{t}(v)$ is the solution to \textsl{the Variational Problem};

\item  $s_{t}(v)$ satisfies the  the following equation in the sense of
distributions
\begin{equation}
\label{Main Equation-1}
(1-L)^{t}s_{t}(v)=\sum_{\nu=1}^{N}\alpha_{\nu}\overline{F_{\nu}},\>\>\>\alpha_{\nu}=\alpha_{\nu}(s_{t}(v)),\>\>
t>d,
\end{equation}
where $\alpha_{1},...,\alpha_{N}$ form a
solution of the $N\times N$ system
\begin{equation}
\label{eq:linsystem}
\sum_{\nu=1}^{N}\beta_{\nu\mu}\alpha_{\nu}=v_{\mu},\>\>\>\alpha_{\nu}=\alpha_{\nu}(s_{t}(v)),\>\>
\mu=1,...,N,
\end{equation}
and
\begin{equation}
\label{eq:linsolution}
\beta_{\nu\mu}=\sum_{j=0}^{\infty}(1+\lambda_{j})^{-t}\overline{F_{\nu}(u_{j})}
F_{\mu}(u_{j}),\>\>\>Lu_{j}=-\lambda_{j}u_{j};
\end{equation}

\item  the Fourier series of $s_{t}(v)$  has  the following form
\begin{equation}
\label{Fourier Series}
s_{t}(v)=\sum_{j=0}^{\infty}c_{j}(s_{t}(v))u_{j},
\end{equation}
where
$$
c_{j}(s_{t}(v))=\left<s_{t}(v),u_{j}\right>=(1+\lambda_{j})^{-t}
\sum_{\nu=1}^{N}\alpha_{\nu}(s_{t}(v))\overline{F_{\nu}(u_{j})}.
$$
\end{enumerate}
\end{theorem}

\begin{remark}
It is important to note that the system (\ref{eq:linsystem}) is always solvable
according to our uniqueness and existence result for the
Variational Problem.

\end{remark}

\begin{remark}
It is also necessary to note that the series (\ref{eq:linsolution}) is absolutely
convergent if $t>d$. Indeed, since functionals $F_{\nu}$
are continuous on the Sobolev space $H_{d/2+\varepsilon}(M)$ we obtain that
for any normalized eigen function $u_{j}$ which corresponds
to the eigen value $\lambda_{j}$ the following inequality holds
true

$$|F_{\nu}(u_{j})|\leq
C(M,F)\|(1-L)^{d/4}u_{j}\|\leq
C(M,F)(1+\lambda_{j})^{d/4},\>\>\>F=\{F_{\nu}\}.
$$
So
$$
|\overline{F_{\nu}(u_{j})}F_{\mu}(u_{j})| \leq
C(M,F)(1+\lambda_{j})^{d/2},
$$
and
$$
|(1+\lambda_{j})^{-t}\overline{F_{\nu}(u_{j})}F_{\mu}(u_{j})|\leq
C(M,F)(1+\lambda_{j})^{(t_{0}-t)}.
$$
It is known  that the series
$$
\sum_{j}\lambda_{j}^{-\tau},
$$
which defines the $\zeta-$function of an elliptic second order
operator, converges if $\tau>d/2$. This implies absolute
convergence of (\ref{eq:linsolution}) in the case $t>d$.

\end{remark}

One can show  that splines provide an optimal approximations to sufficiently smooth  functions.  Namely let $Q(F,f,t,K)$ be the set  of all functions $h$  in
$H_{t}(M)$ such that

\begin{enumerate}

\item  $F_{\nu}(h)=F_{\nu}(f), \nu=1,2,...,N,$

\item  $\|h\|_{t}\leq K,$ for a real $ K\geq \|s_{t}(f)\|_{t}.$

\end{enumerate}

The set $Q(F,f,t,K)$ is  convex,  bounded and closed. 
The following theorem (see \cite{Pes04c}) shows that splines provide an optimal approximations to functions in $Q(F,f,t,K)$.

\begin{theorem}
\label{optim}
The spline $s_{t}(f)$ is the symmetry center of $Q(F,f,t,K)$. This means that for
any $h\in Q(F,f,t,K)$
\begin{equation}
\label{optim-ineq}
  \|s_{t}(f)-h\|_{t}\leq \frac{1}{2} diam \>Q(F,f,t,K).
\end{equation}
\end{theorem}

\subsection{Approximation by pointwise interpolation and approximation}
To formulate our approximation theorem by variational
splines in the case when the set of distributions $F_{i}$ is a set
of  delta functions  we are using notion of a $\rho$-lattice which was introduced in Definition \ref{lattice} and notion of bandlimited functions which was introduced in Definition \ref{bandlimited}.

\begin{theorem}(Approximation Theorem \cite{Pes04a}, \cite{Pes04c})\label{AT}
If $t>d/2+k$ then there exist constants $C(M,t)>0, \>\rho(M)>0$ such that for any
$0<\rho<\rho(M)$, any $\rho$-lattice  $M_{\rho}$, any smooth
function $f$ the following inequality holds true

$$\|(s_{2^{m}d+t}(f)(x)-f(x))\|_{C^{k}(M)}\leq \left(C(M,t)\rho^{2}\right)
^{2^{m}d} \|(1-L)^{2^{m}d+t}f\|, \>\>m=0, 1, ...
$$
and if $f$ is  $\omega$-bandlimited then 
\begin{equation}\label{recon-5}
\|(s_{2^{m}d+t}(f)(x)-f(x))\|_{C^{k}(M)}\leq  \omega^{t}
\left(C(M,t)\rho^{2}\omega\right)^{2^{m}d}\|f\|,
\end{equation}
where $m=0, 1, ... .$
\end{theorem}

The first of these inequalities shows that convergence in $C^{k}(M)$ takes
place when $\rho$ goes to zero and the index $2^{m}d+t$ is fixed.

The second inequality shows that right-hand side goes to
zero for a fixed $\rho$- lattice $M_{\rho}$ as long as
$$
\rho<\left(C(M,t)\omega\right)^{-1/2}
$$
and $m$ goes to infinity.

\subsection{A sampling theorem and a cubature formula}\label{sampl-cub}

Let $l_{x_{\nu}}^{k}$ denote a Lagrangian spline of order $k$, i.e. it  takes value $1$ at the node $x_{\nu}$ and zero at all other points of $M_{\rho}$. 
The last Theorem \ref{AT} can be formulated in the following form.
\begin{theorem}(A Sampling Theorem)
There exists a $c_{0}=c_{0}(M)$  such that for any $\omega>0$ and any $M_{\rho}=\{x_{\nu}\}$ with 
   $\rho =c_{0}\omega^{-1/2}$ the following reconstruction formula holds in $L_{2}(M)$-norm
   \begin{equation}
f=\lim_{l\rightarrow \infty}\sum_{x_{\nu}\in M_{\rho}}f(x_{\nu})l_{x_{\nu}}^{(k)}, \  k\geq d,
\end{equation}
for all  $f\in  {\mathbf E}_{\omega}(L)$.
\end{theorem}

Note, that the right hand side of the last formula involves only values of $f$ on $M_{\rho}$. This statement is, in fact, a generalization of the
classical Sampling Theorem to the case of  compact Riemannian
manifold.

The same result can be used to introduce a family of cubature formulas. To develop such  formulas   we  introduce the notation 
\begin{equation}
\lambda_{\nu}^{(k)}=\int_{\bold M}l_{x_{\nu}}^{(k)}(x)dx,
\end{equation}
where $l_{x_{\nu}}^{k}\in S^{k}(M_{\rho})$ is the Lagrangian spline at the node $x_{\nu}$. 

The next theorem provides a  cubature formula which is exact on variational splines.

\begin{theorem}

There exists a $c_{0}=c_{0}(M)$ such that the following statements hold true.

\begin{enumerate}

\item 

For any $f\in H^{2k}(M)$ one has 
 \begin{equation}
\int_{\bold M}f dx\approx\sum_{x_{\nu}\in M_{\rho}}\lambda_{\nu}^{(k)}f(x{_\nu}),\   k\geq d, \label{ApproxInt333}
\end{equation}
and the error given by the inequality
\begin{equation}
\left|\int_{\bold M}f dx-\sum_{x_{\nu}\in M_{\rho}}\lambda_{\nu}^{(k)}f(x_{\nu})\right|\leq (c_{0}\rho )^{k}\|L^{k/2}f\|,\label{qubSob}
\end{equation}
for $ k\geq d$.  For a fixed function $f$ the right-hand side of (\ref{qubSob}) goes to zero as long as $\rho$ goes to zero.

\item 
The  formula (\ref{ApproxInt333}) is exact for any variational spline $f\in S^{k}(M_{\rho})$ of order $k$ with singularities on $M_{\rho}$.

\end{enumerate}
\end{theorem}

By applying the Bernstein inequality we obtain the following theorem.  This result explains our term "asymptotically correct cubature formulas".
\begin{theorem}There exists a $c_{0}=c_{0}(M)$ such that for any $f\in {\mathbf E}_{\omega}(L)$ one has 
\begin{equation}
\left|\int_{\bold M}f dx-\sum_{x_{\nu}\in M_{\rho}} \lambda_{\nu}^{(k)}f(x_{\nu})\right|\leq (c_{0}\rho\sqrt{\omega} )^{k}\|f\|,\label{qubPW}
\end{equation}
for $ k\geq d$.  If $c_{0}\rho \omega^{-1/2}<1$
the right-hand side in (\ref{qubPW}) goes to zero   for all $f\in {\mathbf E}_{\omega}(L)$ as long as $k$ goes to infinity.
\label{cubPW}
\end{theorem}

\section{Bandlimited and localized Parseval frames on homogeneous manifolds}\label{frames}

In this section  we assume that a manifold $M$ is homogeneous in the sense that it is of the form  $M=G/H,$  where $G$ is a compact Lie group and $H$ is its closed subgroup (see subsection \ref{homman}).  In this situation we consider  spaces of bandlimited functions ${\bf E}_{\omega}(\mathcal{L}),\>\>\>\omega>0,$ with respect to  the Casimir operator $\mathcal{L}$ that was defined in (\ref{Casimir}).
Our goal is to construct a  tight bandlimited and localized frame in the space $L_{2}(M)$.

 Let $g\in C^{\infty}(\mathbb{R}_{+})$ be a monotonic function such that $supp\>g\subset [0,\>  2^{2}], $ and $g(s)=1$ for $s\in [0,\>1], \>0\leq g(s)\leq 1, \>s>0.$ Setting  $G(s)=g(s)-g(2^{2}s)$ implies that $0\leq G(s)\leq 1, \>\>s\in supp\>G\subset [2^{-2},\>2^{2}].$  Clearly, $supp\>G(2^{-2j}s)\subset [2^{2j-2}, 2^{2j+2}],\>j\geq 1.$ For the functions
 $
 \Phi(s)=\sqrt{g(s)}, \>\>\Phi(2^{-2j}s)=\sqrt{G(2^{-2j}s)},\>\>j\geq 1, \>\>\>
 $
 one has 
 $$
 \sum_{j\geq 0}\left(\Phi(2^{-2j}s)\right)^{2}=1, \>\>s\geq 0.
 $$
 Using the spectral theorem for $L$ one  obtains
$$
\sum_{j\geq 0} \Phi^2(2^{-2j}\mathcal{L})f = f,\>\>f \in L_{2}(M),
$$
and taking inner product with $f$ gives
\begin{equation}
\label{norm equality-0}
\|f\|^2=\sum_{j\geq 0}\left< \Phi^2(2^{-2j}\mathcal{L})f,f\right>=\sum_{j\geq 0}\|\Phi(2^{-2j}\mathcal{L})f\|^2 .
\end{equation}
Moreover, since the function $  \Phi(2^{-2j}s)$ has its support in  $
[2^{2j-2},\>\>2^{2j+2}]$ the elements $ \Phi(2^{-2j}\mathcal{L})f $ are bandlimited to  $
[2^{2j-2},\>\>2^{2j+2}]$.

Expanding $f \in L_{2}(M)$ in terms of eigenfunctions of $\mathcal{L}$ we obtain
$$
\Phi({2^{-2j} \mathcal{L}})f= \sum_i \Phi(2^{-2j}\lambda_i)c_i(f) u_i,\>\>\>c_i(f)=\left<f, u_{i}\right>.
$$  
Since for every $j$ function $\Phi(2^{-2j}s)$ is supported in the interval $[2^{2j-2}, 2^{2j+2}]$ the function $\Phi({2^{-2j}\mathcal{L}})f(x),\>\>x\in M, $ is bandlimited and belongs to ${\bf E}_{2^{2j+2}}({\mathcal{L}})$.
But then the function 
$\overline{\Phi({2^{-2j} \mathcal{L}})f(x)}$ is also in ${\bf E}_{2^{2j+2}}({ \mathcal{L}})$.
Since 
$$
|\Phi({2^{-2j}\mathcal{L}})f(x)|^2=\left[\Phi({2^{-2j} \mathcal{L}})f(x)\right]\left[\overline{\Phi({2^{-2j} \mathcal{L}})f(x)}\right],
$$
one can use  Theorem \ref{prodthm} to conclude that  
$
|\Phi({2^{-2j} \mathcal{L}})f|^2\in 
{\bf E}_{4m2^{2j+2}}({ \mathcal{L}}),
$
where $m=\dim G,\>\>{M}=G/H$.

To summarize, we proved, that for every $f\in L_{2}(M)$ we have the following decomposition 
\begin{equation}
\label{addto1sc}
\sum_{j\geq 0} \|\Phi({2^{-2j} \mathcal{L}})f\|^2 = \|f\|^2,\>\>\>\>\>
|\Phi({2^{-2j}\mathcal{L}})f(x)|^2\in 
{\bf E}_{4m2^{2j+2}}({ \mathcal{L}}).
\end{equation}
The next objective is to perform a discretization step. According to our Theorem \ref{cubature}  there exists a constant $c=c(M)>0$ such that for all integer $j$ if  
\begin{equation}
\label{rhoj}
\rho_j = c(4m2^{2j+2}+1)^{-1/2}\sim 2^{-j},\>\>m=\dim G, \>\>M=G/H, 
\end{equation}
then for any  $\rho_{j}$-lattice $M_{\rho_{j}}$ one can find coefficients $b_{j,k}$ with
\begin{equation}
\label{wtest1}
b_{j,k}\sim \rho_j^{d},\>\>\>d=\dim M,
\end{equation}
for which the following exact cubature formula holds
\begin{equation}
\label{cubl2}
\|\Phi({2^{-2j} \mathcal{L}})f\|^2 = \sum_{k=1}^{J_j}b_{j,k}|[\Phi({2^{-2j}\mathcal{L}})f](x_{j,k})|^2,
\end{equation}
where $x_{j,k} \in M_{\rho_j}$, ($k = 1,\ldots,J_j = card\>(M_{\rho_j})$). 

For each $x_{j,k}$ we define the functions
\begin{equation}
\label{vphijkdf}
\psi_{j,k}(y) = \overline{\mathcal{K}^{\Phi}_{2^{-j}}}(x_{j,k},y) = \sum_i \overline{\Phi}(2^{-2j}\lambda_i) \overline{u}_i(x_{j,k}) u_i(y),
\end{equation}
\begin{equation}
\label{phijkdf}
\Psi_{j,k} = \sqrt{b_{j,k}} \psi_{j,k}.
\end{equation}
We find that for all $f \in 
L_2(M)$,
\begin{equation}
\label{parfrm}
\|f\|^2 = \sum_{j,k} |\langle f,\Psi_{j,k} \rangle|^2.
\end{equation}
Moreover, one can show  \cite{gpes-1}, \cite{pg}, \cite{Tay81} that the frame members $\psi_{j,k}$ are strongly localized on the manifold. All together it implies the following statement \cite{gpes-1}, \cite{pg}.
\begin{theorem}\label{main-frames}
If $M$ is a homogeneous manifold, then the set of functions $\{\Psi_{j,k}\},$ constructed in (\ref{vphijkdf})-(\ref{phijkdf}) has the following properties:

\begin{enumerate}

\item  $\{\Psi_{j,k}\},$ is a Parseval frame in the space $L_{2}(M)$. 

\item  Every function $\Psi_{j,k}$ is bandlimited to $[2^{2j-2}, 2^{2j+2}]$.

\item  For  any $N>0$ there exists a $C( N)>0$ such  that uniformly in $j$ and $k$ 
  \begin{equation}\label{localization}
    |\psi_{j,k}(x)|\leq 
    C( N) \frac{2^{dj}}{\max\left(1, \>\>2^{j}d(x,\>x_{j,k})\right)^{N}},\>\>\>j\geq 0.
 \end{equation}

\item  The following reconstruction formula holds
\begin{equation}
\label{recon}
f = \sum_{j\geq 0}^{\infty}\sum_k \langle f,\Psi_{j,k} \rangle \Psi_{j,k} = 
\sum_{j\geq 0}^{\infty}\sum_k b_{j,k} \langle f,\psi_{j,k} \rangle \psi_{j,k},\>\>\>f \in L_2(M),
\end{equation}
with convergence in $L_2(M)$. 

\end{enumerate}

\end{theorem}

By using Theorems \ref{cubature} and \ref{prodthm} one can easily obtain a following {\bf exact}  discrete formula for Fourier coefficients which uses only  samples of $f$ on a sufficiently dense lattice.

\begin{theorem}\label{DFT}
If $M$ is a homogeneous compact manifold then there exists a $c=c(M)>0$ such that for any $\omega>0$, if 
$
\rho_{\omega}=c(\omega+1)^{-1/2},
$  
then
for any $\rho_{\omega}$-lattice $M_{\rho_{\omega}}=\{x_{\nu}\}_{\nu=1}^{N_{\omega}}$ of $M$, there exist positive weights
$
\mu_{\nu}\asymp (\omega+1)^{-d/2}, 
 $
 such that for {\bf every}  function $f$  in ${\bf E}_{\omega}(\mathcal{L})$ the  Fourier coefficients $c_{i}\left( f\right)$  
 $$
 c_{i}(f)=\int_{M}f\overline{u_{i}},\>\>\>\> - \mathcal{L}u_{i}=\lambda_{i}u_{i},\>\>\>\lambda_{i}\leq \omega,  
 $$
  are given by the formulas 
 \begin{equation}\label{DFTF}
c_{i}\left( f\right)=\sum_{\nu=1}^{N_{\omega}}\mu_{\nu}  f(x_{\nu})\overline{u_{i}}(x_{\nu}).
\end{equation}
\end{theorem}

Theorems \ref{main-frames}, \ref{DFT}, \ref{cubature}, and \ref{prodthm} can be used to prove another exact discrete formula for Fourier coefficients which uses frame functions 

 \begin{theorem}\label{D-RepFrame}
 For a compact homogeneous manifold $M$
there exists a constant $c=c(M)>0$ such that for any  natural  $J$ if 
$$
\rho_{J}=c2^{-J}
$$
 then
for any $\rho_{J}$-lattice $M_{\rho_{J}}=\{x^{*}_{\nu}\}_{\nu=1}^{N_{\omega}},$ there exist positive weights 
$$
\mu_{\nu}^{*} \asymp2^{-dJ},\>\>\>d=dim\>M,
$$ 
such that the following  formula holds
\begin{equation}\label{discrrep}
f = \sum_{j=0}^{J}\sum_{k=1}^{n_{j}} \sum_{\nu=1}^{N_{J}}\mu_{\nu}^{*}f(x^{*}_{\nu})\phi_{j,k}(x^{*}_{\nu})\phi_{j,k},\>\>\>f\in {\bf E}_{\omega}({\mathcal L}).
\end{equation}
\end{theorem}

\section{Applications  to Radon transform on $S^{d}$ }\label{Ssplines}

\subsection{Approximate inversion of the spherical Radon transform using generalized splines} We consider approximate inversion  of the Radon transform on $S^{d}$ (see subsection \ref{Radon-spheres}) when only a finite number of integrals over equatorial subspheres is given.  
Let $\{w_{\nu}\}, \nu=1,2,...,N, $ be a finite set of equatorial
subspheres on $S^{d}$ of codimension one and distributions
$F_{\nu}$ are given by
 formulas
 $$
 F_{\nu}(f)=\int_{w_{\nu}}fdx.
 $$
 By solving corresponding variational problem we can find a
spline $s_{t}(f)\in H_{t}(S^{d})$ such that
$$
F_{\nu}(s_{t}(f))=F_{\nu}(f), \nu=1,2,...,N,
$$
and $s_{t}(f)$ minimizes norm $\|(1-\Delta)^{t/2}s_{t}(f)\|$ where $\Delta$ is the Laplace-Beltrami operator in $L_{2}(S^{d})$.

Our Theorem \ref{MainTheorem} in the case of  the spherical Radon transform is summarized in
the following statement.

\begin{theorem}\label{Main-Applic}
For a given symmetric $\rho$- lattice $W=\{w_{\nu}\}$ of
equatorial
 subspheres $\{w_{\nu}\}, \nu=1,2,...,N,$
an even smooth function $f$   and any $t>d/2$ define $s_{t}(f)$
by the formula
$$
s_{t}(f)=\sum_{i,k}c_{i,k}(s_{t}(f))Y^{i}_{k},
$$
where
$$
c_{i,k}(s_{t}(f))=(1+\lambda_{i,k})^{-t}\sum_{\nu=1}^{N}\alpha_{\nu}(s_{t}(f))
\int_{w_{\nu}}Y^{i}_{k}dx,
$$
and
$$
\sum_{\nu=1}^{N}b_{\nu\mu}\alpha_{\nu}(s_{t}(f))=v_{\mu},\>\>\>
v_{\mu}=\int_{w_{\mu}}fdx,\>\>\> \mu=1,2,...,N,
$$
$$b_{\nu\mu}=\sum_{i,k}(1+\lambda_{i,k})^{-t}\int_{w_{\nu}}Y^{i}_{k}dx
\int_{w_{\mu}}Y^{i}_{k}dx.
$$

The function $s_{t}(f)$   has the following properties.

\begin{enumerate}
\item The function $s_{t}(f)$ is even.
\item  Integrals of $s_{t}(f)$ over subspheres $w_{\nu}$ have
prescribed values $v_{\nu}$

$$
\int_{w_{\nu}}s_{t}(f)dx= v_{\nu}, \nu=1,2,...,N;
$$

\item  among all functions that satisfy (2) function $s_{t}(f)$
minimizes
 the Sobolev norm
 $$
 \|(1-\Delta)^{t/2}s_{t}(f)\|=\left(\sum_{\nu=1}^{N}\alpha_{\nu}(s_{t}(f))v_{\nu}\right)^{1/2}.
 $$

\item The function $s_{t}(f)$ is the center of the convex set
$Q(F,f,t,K)$ of all functions $h$ from $H_{t}(S^{d})$ that satisfy both the condition 
(2) and the inequality
$$
\|(1-\Delta)^{t/2}h\|\leq K,
$$
for any fixed $K\geq
\left(\sum_{\nu=1}^{N}\alpha_{\nu}(s_{t}(f))v_{\nu}\right)^{1/2}.$
In other words for any $h\in Q(F, f,t,K)$
$$
\|s_{t}(f)-h\|_{t}\leq\frac{1}{2}diam\> Q(F, f, t, K).
$$
\end{enumerate}

\end{theorem}

To prove item (1) one has to note that because of (\ref{symm}) integrals of $Y_{k}^{i}$ over great circles are zero as long as $k$ is odd. It implies that $c_{i,k}(s_{t}(f))=0$ for every odd $k$ which means that  decomposition of the corresponding spline $s_{t}(f)$   contains only harmonics $Y_{k}^{i}$ with even $k$.

\subsection{ A sampling theorem for the spherical Radon transform of bandlimited functions on $S^{d}$}\label{pointwise}

According to  Theorem \ref{SRT} if $f\in H_{t}^{even}(S^{d})$ then $R f\in H_{t+(d-1)/2}^{even}(S^{d})$. 
Integral of $f$ over a great subsphere  $w_{\nu}$ is the value of $R f$ at a point $x_{\nu}$ where $\{x_{\nu}\}=M_{\rho}\subset (S^{d})$ is a $\rho$-lattice.  For a fixed integer $m\geq 0$ and $t>d/2$ we  apply Theorem \ref{MainTheorem} to the manifold $S^{d}$ and the set of distributions $F=\{\delta_{x_{\nu}}\}$ where
$$
\delta_{x_{\nu}}(g)=g(x_{\nu}),\>\>\>g\in C^{0}(S^{2}\times S^{2}),\>\>x_{\nu}\in M_{\rho}\subset S^{d}
$$
to construct  spline  $s_{\tau}(R f)$ with
$
\tau=2^{m}d+t+(d-1)/2
$
 which interpolates $Rf$ on $M_{\rho}=\{x_{\nu}\}$ and minimizes functional \begin{equation}\label{functional20}
u\rightarrow \|(1- \Delta)^{\tau/2}u\|,
\end{equation}
where $\Delta$ is the Laplace-Beltrami operator in $L_{2}(S^{d})$.
 Fourier coefficients of $s_{\tau}(R f)$ with respect to the basis $\left\{Y_{k}^{i}\right\}$ can be obtained by using Theorem \ref{MainTheorem}. 
Let $\widehat{s}_{\tau}(R f)$ be the orthogonal projection (in the norm of $H_{t+(d-1)/2}(S^{d})$) of $s_{\tau}(R f)$ onto subspace 
$
H_{\tau}^{even}(S^{d}).
$
It means that $\widehat{s}_{\tau}(R f)$ has a representation of the form
\begin{equation}\label{shat}
\widehat{s}_{\tau}(Rf)(x)=\sum_{k}\sum_{i} c_{i}^{2k}(R f)Y_{2k}^{i}(x),\>\>\>x\in S^{d},
\end{equation}
where
$$
c_{i}^{k}(R f)=c_{i}^{k}(R f; m, t)=\int_{ S^{d}}s_{\tau}(R f)(x)\overline{Y_k^i(x)}dx 
$$
are Fourier coefficients of $s_{\tau}(Rf)$. 
Applying (\ref{inverse}) we obtain that the following function defined on $S^{d}$ 
\begin{equation}\label{bigS}
S_{\tau}(f)=R^{-1}\left(\widehat{s}_{\tau}(R f)\right)
\end{equation}
has a representation 
$$
S_{\tau}(f)=\frac{\sqrt{\pi}}{\Gamma((d+1)/2)}\sum_{k}\sum_{i} \frac{c^{i}_{2k}(R f)}{r_{2m}}Y_{2k}^{i},
$$
where $r_{2k}$ are  defined in  (\ref{r}).
Let us stress that these functions do not interpolate $f$ in any sense. However, as it will be shown they can be used to approximate $f$.

In the following theorem we assume that a $\rho$-lattice $M_{\rho}=\{x_{\nu}\}$ is a subset of points  on the sphere $S^{d}$ and $M_{\rho}$  is dual to a collection of great subspheres $\{w_{\nu}\}$. we also assume that functions $S_{\tau}(f)$ constructed using vales of $Rf$ on $M_{\rho}$.
One can prove the following Theorem (see \cite{Pes04c}).

\begin{theorem}\label{last-thm555}
If $t>d/2$ then there exists a constant $C=C(d,t)>0$ such that for  any $\rho$-lattice  $M_{\rho}=\{x_{\nu}\}\subset S^{d}$ with sufficiently small $\rho>0$  and any sufficiently  smooth
function $f$ on $S^{d}$ the following inequality holds true
\begin{equation}\label{Th62}
\|\left(S_{\tau}(f)-f\right)\|_{H_{t}(S^{d})}\leq
 2(C\rho^{2d})^{2^{m-1}}\|Rf\|_{H_{\tau}(S^{d})},\>\>\>\tau=2^{m}d+t+(d-1)/2,
\end{equation}
for any $m=0, 1, ... \>\>.\>\>\>$ In particular, if a natural $k$ satisfies the inequality $t>k+d/2$, then 
\begin{equation}
\|S_{\tau}( f)-  f\|_{C^{k}(S^{d})}\leq 
2\left(C\rho^{2d}\right)^{ 2^{m-1}} \|R f\|_{H_{\tau}(S^{d})}
\end{equation}
for any $m=0, 1, ... .$ 
\end{theorem}

For an $\omega>0$ let us consider the subspace ${\bf E}_{\omega}^{even}(S^{d})$ of even $\omega$-bandlimited functions on $S^{d}$. Clearly, this subspace is invariant under $R$. 

Note, that for functions in ${\bf E}_{\omega}^{even}(S^{d})$
 the following Bernstein-type inequality holds 
 $$
\|(1-\Delta)^{s}Rf\|_{L^{2}(S^{d})}  \leq (1+\omega)^{s}\|Rf\|_{L^{2}(S^{d})} .
$$
As a consequence of the previous Theorem we obtain the next one (see \cite{Pes04c}.
\begin{theorem}(Sampling Theorem For Radon Transform)\label{SSTh}.
\label{SamplingT}
If $t>d/2$ then there exist  constant $C=C(d, t)>0$ such that for  any $\rho$-lattice  $M_{\rho}=\{x_{\nu}\}\subset S^{d}$ with sufficiently small $\rho>0$  and any  $f\in {\bf E}_{\omega}^{even}(S^{d})$
  one has the estimate
\begin{equation}
\|\left(S_{2^{m}d+t+(d-1)/2}(f)-f\right)\|_{H_{t}(S^{d})}\leq
$$
$$
2(1+\omega)^{t/2+(d-1)/2}\left(C\rho^{2}(1+\omega)\right)^{2^{m-1}d} \| Rf\|_{L_{2}(S^{d})},
 \end{equation}
for any $m=0, 1, ... \>\>.\>\>\>$ In particular, if a natural $k$ satisfies the inequality $t>k+d/2$, then 
\begin{equation}
\|S_{2^{m}d+t+(d-1)/2}( f)-  f\|_{C^{k}(S^{d})}\leq
$$
$$
 2(1+\omega)^{t/2+(d-1)/2}\left(C\rho^{2}(1+\omega)\right)^{2^{m-1}d} \| Rf\|_{L_{2}(S^{d})},
\end{equation}
for any $m=0, 1, ... .$ 
\end{theorem}
This Theorem \ref{SSTh} shows that if 
$
\rho<c\omega^{-1/2}
$
then every $f\in {\bf E}_{\omega}(S^{d})$ is completely determined by a finite  set of values 
$$
Rf(x_{\nu})=\int_{w_{\nu}} f,\>\>\>\>\>\>x_{\nu}\in M_{\rho}.
$$
Moreover, it shows that $f$  can be reconstructed as a limit (when $m$ goes to infinity) of functions $S_{2^{m}d+t+(d-1)/2}( f)$ which were  constructed by using only the values of  the Radon transform $R f(x_{\nu})$.

\subsection{ Exact formulas for Fourier coefficients  of  a bandlimited function $f$  on $S^{d}$ from  a finite number of samples of  $R f$}\label{exactFCS}
 Theorem \ref{DFT} can be used to obtain a discrete inversion formula for $R$.

\begin{theorem}(Discrete Inversion Formula)\label{SamplingTh}
There exists a $c=c(M)>0$ such that for any $\omega>0$, if 
$
\rho_{\omega}=c\omega^{-1/2},
$  
then
for any $\rho_{\omega}$-lattice $M_{\rho_{\omega}}=\{x_{\nu}\}_{\nu=1}^{m_{\omega}}$ of $S^{d}$, there exist positive weights
$
\mu_{\nu}\asymp \omega^{-d/2}, 
 $
 such that for every  function $f$  in ${\bf E}_{\omega}(S^{d})$ the  Fourier coefficients $c^{i}_{k}\left(R f\right)$  of its  Radon transform, i.e.
 $$
Rf(x)=\sum_{i,k} c^{i}_{k}\left( R f\right)\times Y^{i}_{k}(x),\>\>\>\>\>\>\>\>\>k(k+1)\leq \omega,\>\>\>\>\>x\in S^{d},
 $$
  are given by the formulas 
 \begin{equation}
c^{i}_{k}\left(R f\right)=\sum_{\nu=1}^{m_{\omega}}\mu_{\nu}\times \left(R f\right)(x_{\nu})\times Y^{i}_{k}(x_{\nu}).
\end{equation}
The function $f$ can be reconstructed by means of  the formula
\begin{equation}\label{Exact555}
f=\frac{\sqrt{\pi}}{\Gamma((d+1)/2)}\sum_{k}\sum_{i}\frac{c^{i}_{k}\left(R f\right)}{r_{k}}\times Y_{k}^{i},
\end{equation}
in which $k$ runs over all natural even numbers  such that $k(k+1)\leq \omega$ and $r_{k}$ are defined in (\ref{r}).
\end{theorem}

\begin{proof}An application of Theorem \ref{prodthm-2} shows that if $k(k+1)\leq \omega$ then  every product 
$
 Y^{i}_{k}\overline{Y^{j}_{k}}$, where $k(k+1)\leq \omega$ belongs to   ${\bf E}_{2\omega}(S^{d})$.

By  Theorem \ref{cubature} there exists  a  positive constant $c=c(M)$,    such  that if  $\rho_{\omega}=c\omega^{-1/2}$, then
for any $\rho_{\omega}$-lattice $M_{\rho_{\omega}}=\{x_{\nu}\}^{m_{\omega}}_{\nu=1}$  on  $S^{d}$ 
there exist  a set of positive weights 
$
\mu_{\nu}\asymp \omega^{-d/2}
$
such that 
\begin{equation}
c^{i}_{k}\left( R f\right)=\int_{S^{d}} \left(R f\right)(x)\times \overline{Y^{i}_{k}(x)} dx=
\sum_{\nu=1}^{m_{\omega}}\mu_{\nu}\times \left(R f\right)(x_{\nu})\times \overline{Y^{i}_{k}(x_{\nu})}.
\end{equation}
Thus,
$$
\left(Rf\right)(x)=\sum_{k,i}c^{i}_{k}\left( R f\right)\times \overline{Y^{i}_{k}(x)}.
$$
Now the reconstruction formula of Theorem \ref{Recon} implies the formula (\ref{Exact555}).

\end{proof}

\bigskip

Without going to details we just mention that similar results can be obtained for the hemispherical Radon transform on $S^{d}$ (see \cite{Pes04c}) and for the case of Radon transform of functions on domains see \cite{Pes12}.

\section{Applications to the Radon transform on $SO(3)$}\label{group-applic}

\subsection{Approximate inversion of the Radon transform on $SO(3)$ using  generalized splines}

Let $\{(x_{1}, y_{1}),...,(x_{N}, y_{N})\}$ be a set of pairs of points from $SO(3)$. In what follows we have to assume that our {\bf Independence Assumption} (\ref{independence}) holds. It takes now the following form: there are smooth functions $\phi_{1},...,\phi_{N}$ on $SO(3)$  with
$$
\mathcal{R} \phi_{\mu}(x_{\nu}, y_{\nu})=\delta_{\nu\mu}.
$$
But it is obvious that for this condition to satisfy it is enough to assume that submanifolds $\mathcal{M}_\nu=x_\nu SO(2) y_\nu^{-1}\subset SO(3)$ are pairwise different (not necessarily disjoint).

Let $f$ be  a function in $H_{t}(SO(3)),\>\>\>$ $t> \frac{1}{2}dim \> SO(3)=3/2$ and
\begin{equation}\label{constraint}
v_{\nu}=\int_{\mathcal{M}_{\nu}}f,\>\>\>\nu=1,...,N,\>\>\>v=\{v_{\nu}\}.
\end{equation}
According to Definition \ref{interpolationdef} we use notation $s_{t}(f)=s_{t}(v)$ for a function in $ H_{t}(SO(3))$  such that for $\mathcal{M}_{\nu}=x_{\nu}  SO(2) y_{\nu}^{-1}$  it satisfies (\ref{constraint}) and minimizes the functional  
\begin{equation}\label{functional}
 u\rightarrow \|(1- 4\Delta_{SO(3)} )
^{t/2}u\|.
\end{equation}

In this situation the results of section \ref{splines} can be summarized in the following statement which was proved in \cite{BEP}.

\begin{theorem}
Let $\{(x_{1}, y_{1}),...,(x_{N}, y_{N})\}$ be a subset of  $SO(3)\times SO(3)$ such  that submanifolds $\mathcal{M}_\nu=x_\nu SO(2) y_\nu^{-1}\subset  SO(3),\>\>\nu=1,...,N,$ are pairwise different.

For a function $f$ in $H_{t}(SO(3)),\>\>\>$ $t>3/2,$ and a vector of measurements $v=\left(v_{\nu}\right)_{1}^{N}$ in (\ref{constraint}) 
 the solution of a constrained variational problem (\ref{constraint})-(\ref{functional}) is given by 
\begin{equation}
\label{splineSO}
s_{t}(f)=\sum_{k=0}^{\infty}\sum_{i,j=1}^{2k+1}c_{ij}^{k}(s_{t}(f))\mathcal{T}_{ij}^{k}=\sum_{k=0}^{\infty}trace\left(c_{k}(s_{t}(f))\mathcal{T}^{k}\right),
\end{equation}
where $\mathcal{T}_{ij}^{k}$ are the Wigner polynomials. The Fourier coefficients $c_{k}(s_{t}(f))$ of the solution are given by their matrix entries 
\begin{equation}
c_{ij}^{k}(s_{t}(f))=\frac{4\pi}{(2k+1)(1+k(k+1))^{t}}\sum_{\nu=1}^{N}\alpha_{\nu}(s_{t}(f) Y_k^i(x_{\nu})\overline{Y_k^j(y_{\nu})},
\end{equation}
where $\alpha(s_{t}(f))=\left(\alpha_{\nu}(s_{t}(f))\right)_{1}^{N}\in \mathbb{R}^{N}$ is the solution of 
\begin{equation}
\beta\alpha(s_{t}(f))=f,
\end{equation} 
with $\beta\in \mathbb{R}^{N\times N}$ given by 

\begin{align}
    \beta_{\nu\mu}= \sum_{k=0}^\infty (1+k(k+1))^{-t} C_k^{\frac12}(x_\nu\cdot y_\nu) C_k^{\frac12}(x_\mu\cdot y_\mu),
\end{align}
where  the Gegenbauer polynomials $C_k^{\frac12}$ are given by the formulas 
\begin{equation}
 \mathcal C_k^{\frac12} (x\cdot y)
  = \frac{4\pi}{2k+1} \sum_{i=1}^{2k+1} Y^i_k(x)\overline{Y^i_k(y)}
\end{equation}
 for all $x,y\in S^2$ and $k=0, 1, 2...\>\>$.
The function $s_{t}(f)\in H_{t}(SO(3))$ has the following properties:

\begin{enumerate}

\item $s_{t}(f)$ has the prescribed set of measurements $v\!=\left(v_{\nu}\right)_{1}^{N}$ at points $((x_{\nu}, y_{\nu}))_{1}^{N}$;

\item it minimizes the functional (\ref{functional});

\item  the solution (\ref{splineSO}) is optimal in the sense that for every sufficiently large $K>0$ it is the symmetry center of the convex bounded closed set of all functions $h$  in $H_{t}(SO(3))$  with $\|h\|_{t}\leq K$ which have the same set of measurements $v=\left(v_{\nu}\right)_{1}^{N}$ at points $((x_{\nu}, y_{\nu}))_{1}^{N}$.

\end{enumerate}

\end{theorem}

\subsection{ A sampling theorem for Radon transform of bandlimited functions on $SO(3)$}

According to  Theorem \ref{mapRT} if $f\in H_{t}(SO(3))$ then $\mathcal{R} f\in H_{t+1/2}^{\Delta}(S^{2}\times S^{2})$. 
Integral of $f$ over the manifold $x_{\nu}SO(2) y_{\nu}^{-1}$ is the value of $\mathcal{R} f$ at $(x_{\nu}, y_{\nu})$ where $\{(x_{\nu}, y_{\nu})\}=M_{\rho}\subset (S^{2}\times S^{2})$ is a $\rho$-lattice.  Note that dimension $d$ of the manifold $S^{2}\times S^{2}$ is four. For a fixed natural $m$ and $t>3/2$ we  apply Theorem \ref{MainTheorem} to the manifold $S^{2}\times S^{2}$ and the set of distributions $F=\{\delta_{\nu}\}$ where
$$
\delta_{\nu}(g)=g(x_{\nu}, y_{\nu}),\>\>\>g\in C^{0}(S^{2}\times S^{2}),\>\>(x_{\nu},y_{\nu})\in M_{\rho}\subset S^{2}\times S^{2},
$$
to construct  spline  $s_{2^{m}d+(t+1)}(\mathcal{R} f)$  which interpolates $\mathcal{R} f$ on $M_{\rho}=\{(x_{\nu}, y_{\nu})\}$ and minimizes functional \begin{equation}\label{functional2}
u\rightarrow \|(1- 2\Delta_{S^{2}\times S^{2}} )^{2^{m-1}d+(t+1)/2}u\|,\>\>\>d=4.
\end{equation}
 Fourier coefficients of $s_{2^{m}d+(t+1)}(\mathcal{R} f)$ with respect to the basis $\left\{Y_{k_{1}}^{i}Y_{k_{2}}^{j}\right\}$ can be obtained by using Theorem \ref{MainTheorem}. 
Let $\widehat{s}_{2^{m}d+(t+1)}(\mathcal{R} f)$ be the orthogonal projection (in the norm of $H_{t}(S^{2}\times S^{2})$) of $s_{2^{m}d+(t+1)}(\mathcal{R} f)$ onto subspace 
$
H_{2^{m}d+(t+1)}^{\Delta}(S^{2}\times S^{2}).
$
It means that $\widehat{s}_{2^{m}d+(t+1)}(\mathcal{R} f)$ has a representation of the form
\begin{equation}\label{shat}
\widehat{s}_{2^{m}d+(t+1)}(\mathcal{R} f)(x,y)=\sum_{k}\sum_{ij} c_{ij}^{k}(\mathcal{R} f)Y_k^i(x)\overline{Y_k^j(y)},\>\>\>(x,y)\in S^{2}\times S^{2},
\end{equation}
where
$$
c_{ij}^{k}(\mathcal{R} f)=\int_{ S^{2}\times S^{2}}s_{2^{m}d+(t+1)}(\mathcal{R} f)(x,y)\overline{Y_k^i(x)}Y_k^j(y)dx dy
$$
are Fourier coefficients of $s_{2^{m}d+(t+1)}(\mathcal{R} f)$. 
Applying (\ref{basis-action1})  we obtain that the following function defined on $SO(3)$ 
\begin{equation}\label{bigS}
S_{2^{m}d+(t+1)}(f)(x)=\mathcal{R}^{-1}\left(\widehat{s}_{2^{m}d+(t+1)}(\mathcal{R} f)\right)(x),\>\>\>d=4
\end{equation}
has a representation 
$$
S_{2^{m}d+(t+1)}(f)(x)=\sum_{k}\sum_{ij} \frac{2k+1}{4\pi}c_{ij}^{k}(\mathcal{R} f)\mathcal{T}^{k}_{ij}(x),\>\>\>d=4.
$$
Let us stress that these functions do not interpolate $f$ in any sense. However, the following approximation results were proved in \cite{BEP}.

\begin{theorem}\label{last-thm}
If $t>3/2$ then there exist a constant $C=C(t)>0$ such that for  any $\rho$-lattice  $M_{\rho}=\{(x_{\nu}, y_{\nu})\}\subset S^{2}\times S^{2}$ with sufficiently small $\rho>0$  and any sufficiently  smooth
function $f$ on $SO(3)$ the following inequality holds true
$$
\|\left(S_{\tau}(f)-f\right)\|_{H_{t}(SO(3))}\leq
C_{1}(m)\rho^{ 2^{m+2}} \|\mathcal{R} f\|_{H_{\tau}(S^{2}\times S^{2})},.
$$
for $\tau=2^{m+2}+(t+1)$ and any $m=0, 1, ... \>\>.\>\>\>$ In particular, if a natural $k$ satisfies the inequality $t>k+3/2$, then 
\begin{equation}
\|S_{\tau}( f)-  f\|_{C^{k}(SO(3))}\leq 
C_{1}(m)\rho^{ 2^{m+2}}  \|\mathcal{R} f\|_{H_{\tau}(S^{2}\times S^{2})}
\end{equation}
for any $m=0, 1, ... .$ 
\end{theorem}

For an $\omega>0$ let us consider the space ${\bf E}_{\omega}(SO(3))$ of $\omega$-bandlimited functions on $SO(3)$ i.e. the span of all Wigner functions $T_{ij}^{k}$ with $k(k+1)\leq \omega$. As the formulas (\ref{eigenvalues}) and (\ref{basis-action1}) show the Radon transform of such function is $\omega$-bandlimited on $S^{2}\times S^{2}$ in the sense its Fourier expansion involves only functions $Y^{i}_{k}\overline{Y^{j}_{k}}$ which are eigenfunctions of $\Delta_{S^{2}\times S^{2}}$ with eigenvalue $-2k(k+1)\geq -2\omega$. Let ${\bf \mathcal{E}}_{\omega}(S^{2}\times S^{2})$ be the span of $Y_k^i(\xi)\overline{Y_k^j(\eta)}$ with $k(k+1)\leq  \omega$. Thus
 \begin{equation}\label{band-band}
  \mathcal{R}: {\bf E}_{\omega}(SO(3))\rightarrow{\bf  \mathcal{E}}_{\omega}(S^{2}\times S^{2}).
\end{equation}
For $f\in {\bf E}_{\omega}(SO(3))$ the following Bernstein-type inequality holds $$
\|(1-2\Delta_{S^{2}\times S^{2}})^{\tau}\mathcal{R} f\|_{L^{2}(S^{2}\times S^{2})}  \leq (1+4\omega)^{\tau}\|\mathcal{R} f\|_{L^{2}(S^{2}\times S^{2})}.
$$
For the proof of the next theorem we refer to \cite{BEP}.
\begin{theorem}(Sampling Theorem For Radon Transform).
\label{SamplingT}
If $t>3/2$ then there exist a constant $C=C(t)>0$ such that for  any $\rho$-lattice  $M_{\rho}=\{(x_{\nu}, y_{\nu})\}\subset S^{2}\times S^{2}$ with sufficiently small $\rho>0$  and any  $f\in {\bf E}_{\omega}(SO(3))$
  one has the estimate
\begin{equation}
\|\left(S_{\tau}(f)-f\right)\|_{H_{t}(SO(3))}\leq
$$
$$
2(1+4\omega)^{(t+1)/2}\left[C\rho^{2}(1+4\omega)\right]^{2^{m+1}}\|\mathcal{R}f\|_{L_{2}(S^{2}\times S^{2})},
 \end{equation}
for $\tau=2^{m+2}+(t+1)$ and any $m=0, 1, ... \>\>.\>\>\>$ In particular, if a natural $k$ satisfies the inequality $t>k+3/2$, then 
$$
\|S_{\tau}( f)-  f\|_{C^{k}(SO(3))}\leq
2(1+4\omega)^{(t+1)/2}\left[C\rho^{2}(1+4\omega)\right]^{2^{m+1}}\|\mathcal{R}f\|_{L_{2}(S^{2}\times S^{2})},
$$
for any $m=0, 1, ... .$ 
\end{theorem}
This Theorem \ref{SamplingT} shows that if 
$
\rho<C(1+\omega)^{-1/2}
$
then every $f\in {\bf E}_{\omega}(SO(3))$ is completely determined by a finite  set of values 
$$
\mathcal{R}f(x_{\nu}, y_{\nu})=\int_{M_{\nu}} f,
$$
where $M_{\nu}=x_{\nu}SO(2) y_{\nu}^{-1}\subset SO(3),\>\>\>\{(x_{\nu}, y_{\nu})\}=M_{\rho}\subset S^{2}\times S^{2}.$
Moreover, it shows that $f$  can be reconstructed as a limit (when $m$ goes to infinity) of functions $S_{2^{m+2}+t}( f)$ which were  constructed by using only the values of  the Radon transform $\mathcal{R} f(x_{\nu}, y_{\nu})$.

\subsection{ Exact formulas for Fourier coefficients  of  a bandlimited function $f$  on $SO(3)$ from  a finite number of samples of  $ \mathcal{R} f$}\label{Discrete}

Let $M_{\rho}=\{(x_{\nu}, y_{\nu})\}$ be  a metric
$\rho$-lattice of  $S^{2}\times S^{2}$.
In what follows $\mathcal{E}_{\omega}(S^{2}\times S^{2})$ will denote  the span in the space $L^{2}(S^{2}\times S^{2})$ of all $Y_k^i\overline{Y_k^j}$  with $k(k+1)\leq \omega$. Theorem \ref{DFT} implies the following exam discrete reconstruction formula which uses only  samples of $\mathcal{R} f$ on a sufficiently dense lattice.

\begin{theorem}(Discrete Inversion Formula \cite{BP})\label{SamplingTh}
There exists a $c=c(M)>0$ such that for any $\omega>0$, if 
$
\rho_{\omega}=c\omega^{-1/2},
$  
then
for any $\rho_{\omega}$-lattice $M_{\rho_{\omega}}=\{(x_{\nu}, y_{\nu})\}_{\nu=1}^{m_{\omega}}$ of $S^{2}\times S^{2}$ , there exist positive weights
$
\mu_{\nu}\asymp \omega^{-2}, 
 $
 such that for every  function $f$  in ${\bf E}_{\omega}(SO(3))$ the  Fourier coefficients $c_{i,j}^{k}\left( \mathcal{R} f\right)$  of its  Radon transform, i.e.
 $$
  \mathcal{R} f(x,y)=\sum_{i,j,k} c_{i,j}^{k}\left( \mathcal{R} f\right)\times Y^{i}_{k}(x)\overline{Y^{j}_{k}}(y),\>\>\>\>\>\>\>\>\>k(k+1)\leq \omega,\>\>\>\>\>(x,y)\in S^{2}\times S^{2},
 $$
  are given by the formulas 
 \begin{equation}
c_{i,j}^{k}\left(\mathcal{R} f\right)=\sum_{\nu=1}^{m_{\omega}}\mu_{\nu}\times \left(\mathcal{R} f\right)(x_{\nu},y_{\nu})\times Y^{i}_{k}(x_{\nu})\overline{Y^{j}_{k}}(y_{\nu}).
\end{equation}
The function $f$ can be reconstructed by means of  the formula
\begin{equation}\label{Exact}
f(g)=\sum_{k}\sum_{i,j}^{2k+1}\frac{(2k+1)}{4\pi}\times c_{i,j}^{k}\left( \mathcal{R} f\right)\times \mathcal{T}_{k}^{i,j}(g),\>\>\>\>\>g\in SO(3),
\end{equation}
in which $k$ runs over all natural numbers  such that $k(k+1)\leq \omega$.
\end{theorem}

					\section{Conclusion}
					
	During last few decades there was a strong demand for  effective tools to perform interpolation, approximation, harmonic  and multiscale analyses on such manifolds as  the unit spheres $S^{2}$ and $S^{3}$ and  the rotation group of $SO(3).$			
		
In our paper we develop a  theory of variational interpolating splines 
 on  compact Riemannian manifolds.
Our consideration includes in particular such problems as
interpolation of a function by its values on a discrete set of
points and interpolation by
 values of integrals over
  a family of submanifolds.
The existence and uniqueness of interpolating variational
 spline on a Riemannian manifold is proven.
 Optimal properties of such splines are shown.
 The explicit formulas of
  variational splines in terms of the eigenfunctions of an appropriate elliptic differential  operator are found.
 It is also shown that in the case of interpolation on discrete sets of points
 variational splines converge to a function in $C^{k}$ norms on manifolds.  Variational  splines are used to establish a generalization of the Classical Sampling Theorem and to develop cubature formulas on manifolds.

In a different venue  we extend the well-developed field of time-frequency analysis based on   frames  and Shannon sampling  to compact Riemannian manifolds by constructing bandlimited localized Parseval frames on compact manifolds which have many symmetries. 

Applications of these results to the  Radon
transforms on the unit sphere $S^{d}$  and on the group of rotations $SO(3)$ are given.

The present study expands the important  field of  multiresolution analysis based on splines and frames  from Euclidean spaces to compact Riemannian manifolds thus giving the means for modeling and analyzing various important complex phenomena.

\end{document}